\documentclass[12pt]{amsart}

\usepackage{amsmath}
\usepackage{amssymb} 
\usepackage[all,cmtip]{xy}
\usepackage{url}
\usepackage{ulem,enumitem}
\usepackage{bm}
\newcommand{\Z}{{\mathbb Z}} 
\newcommand{\Q}{{\mathbb Q}}

\newcommand{\cO}{{\mathcal O}}

\newcommand{\N}{\mathbb{N}}
\newcommand{\SL}{\operatorname{SL}}
\newcommand{\p}{{\mathfrak p}}

\newtheorem{thm}{Theorem}
\newtheorem{lemma}[thm]{Lemma}

\newtheorem{prop}[thm]{Proposition}

\theoremstyle{remark}

\newtheorem{rem}{Remark}

\numberwithin{equation}{section}
\numberwithin{thm}{section}
\allowdisplaybreaks

\renewcommand{\Im}{\operatorname{Im}}

\begin{document}  

\author{P. Guerzhoy}
\address{ Department of Mathematics, University of Hawaii, 2565 McCarthy Mall, Honolulu, HI,  96822-2273  }
\email{pavel@math.hawaii.edu}
\author{Ben Kane}
\address{Department of Mathematics, University of Hong Kong, Pokfulam, Hong Kong}
\email{bkane@hku.hk}
\title[A special case of Siegel's mass formula]{A very special case of Siegel's mass formula and Hecke operators}
\thanks{The research of the second author is supported by grant project numbers 17316416, 17301317, and 17303618 of the Research Grants Council of Hong Kong SAR. Part of the research was also conducted while the second author was supported by grant project numbers 273000314 and 17302515 of the Research Grants Council of Hong Kong SAR}
\date{\today}
\begin{abstract}
We make use of Hecke operators and arithmetic of imaginary quadratic fields to derive 
an explicit version of a special case of Siegel's mass formula.
\end{abstract}
\maketitle
 
\section{Introduction} \label{intro}

In this paper, we investigate representations of integers by positive-definite integral quadratic forms. Ranging back to Lagrange's proof that every integer may be written as the sum of four squares, this question has a long and storied history. Letting $r(Q,n)$ denote the number of representations of $n\in\N_0:=\N\cup\{0\}$ by a positive-definite integral quadratic form $Q$ of rank $\ell$, it is now well-known (see for example \cite[Proposition 2.1]{Shimura}) that the generating function ($q:=e^{2\pi i \tau}$, with $\Im (\tau)>0$ throughout)
\[
\theta_Q(\tau):=\sum_{n\geq 0} r(Q,n) q^n
\]
is a modular form of weight $\ell/2$ on a congruence subgroup with some (quadratic) character. Using this fact, one may naturally decompose $\theta_Q$ into an Eisenstein series $E_Q$ and a cusp form $f_Q$; namely, we have 
\begin{equation}\label{eqn:thetaQdecompose}
\theta_Q=E_Q+f_Q.
\end{equation}
Using \eqref{eqn:thetaQdecompose} together with bounds on the asymptotic growth of the Fourier coefficients of $E_Q$ and $f_Q$, Tartakowsky \cite{Tartakowsky1,Tartakowsky2} proved that when $\ell\geq 5$ (extended under some mild conditions to $\ell=4$ by Kloosterman \cite{Kloosterman}, who exploited the modularity in his usage of the Hardy-Littlewood Circle Method, and to $\ell=3$ by Duke--Schulze-Pillot \cite{DukeSchulzePillot} under more serious restrictions) every sufficiently large integer which is locally represented (i.e., modulo any power of any prime, and hence, by the Chinese remainder theorem, modulo any integer) is represented by $\theta_Q$; moreover, the $n$th coefficient of $E_Q$ is positive if and only if $n$ is locally represented. Given Tartakowsky's result, one is naturally led to ask the following question: given $Q$, can one construct $E_Q$ explicitly? This question was essentially resolved by Siegel \cite{Siegel,Siegel1,Siegel2} and subsequent work of Weil \cite{Weil}.

Let $w_Q$ denote the number of isometries from $Q$ to itself (the so-called \begin{it}automorphs\end{it} of $Q$) and let $\mathcal{G}=\mathcal{G}(Q)$ denote a set of representatives of the isometry classes of quadratic forms in a \begin{it}genus\end{it} $g=\text{gen}(Q)$ (i.e., a set of quadratic forms  which  are isometric over the adeles, but not necessarily over the integers).
Let
\[
E_g:= \frac{1}{\sum_{Q\in \mathcal{G}} w_Q^{-1}} \sum_{Q\in \mathcal{G}}\frac{\theta_{Q}}{w_Q}.
\]
\begin{it}Siegel's mass formula\end{it} (also known as the \begin{it}Siegel--Weil average\end{it}) states that 
\begin{equation}\label{eqn:SiegelWeil}
E_g=E_{Q_0}=\theta_{\operatorname{gen}(Q_0)}:= \frac{1}{\sum_{Q\in \mathcal{G}(Q_0)} w_Q^{-1}} \sum_{Q\in \mathcal{G}(Q_0)}\frac{\theta_{Q}}{w_Q}.
\end{equation}
This formula implies, in particular, that $E_g$ belongs to the Eisenstein series subspace. Moreover, while $\theta_Q$ (therefore $E_Q$) is clearly class-invariant, for $Q_1$ and $Q_2$ in the same genus but not in the same class, it may well be the case that $\theta_{Q_1} \neq \theta_{Q_2}$. The Siegel--Weil average \eqref{eqn:SiegelWeil} implies that the projection to the Eisenstein series 
is genus-invariant, not merely class-invariant: $E_{Q_1}=E_{Q_2}=E_g$ with $g=\text{gen}(Q_1)=\text{gen}(Q_2)$.

\begin{rem}
\noindent
\begin{enumerate}[leftmargin=*]
\item 
Siegel actually stated \eqref{eqn:SiegelWeil} in much more generality in \cite{Siegel}, investigating representations of quadratic forms by quadratic forms. In other words, for another quadratic form $\mathfrak{Q}$ of rank $m\leq \ell$ with Gram matrix $A_{\mathfrak{Q}}$, Siegel was interested in determining the set of $\ell \times m$ matrices $M$ for which 
\[
A_{\mathfrak{Q}} = M^T A_{Q} M.
\]
A corresponding generating function for the number of such representations for all $\mathfrak{Q}$ of rank $m$ is a higher-dimensional modular forms that is now known as a \begin{it}Siegel modular form\end{it} (of weight $\ell/2$ and degree $m$) and the weighted average in the sense of \eqref{eqn:SiegelWeil} is indeed the Siegel--Eisenstein series component of this Siegel modular form. In the case $m=1$, one notes that the unary quadratic forms $\mathfrak{Q}(x)=nx^2$ are in one-to-one correspondence with the integers $n\in\N$, and the generating function precisely becomes $\theta_Q$. 
\item 
After determining representatives for the genus of $Q_0$ (see \cite{SchulzePillotAlgorithm} for an algorithm yielding these for $\ell=3,4$), one can then compute the Fourier coefficients of the theta functions on the right-hand side of \eqref{eqn:SiegelWeil} up to the bound from the valence formula and use linear algebra to determine the coefficients of $E_{Q_0}$ explicitly in practice (at least algorithmically speaking). 
\item
The coefficients of cusp forms behave in a more ``mysterious'' and ``erratic'' way than those of Eisenstein series (e.g., 
Knopp, Kohnen and Pribitkin \cite{KnoppKohnenPribitkin} (integral weight case) and Kohnen, Lau, and Wu \cite{KohnenLauWu} (half-integral weight) have shown that when they have real coefficients 
they alternate in sign infinitely often, 
it is still unknown whether any of the coefficients of the unique weight $12$ normalized newform $\Delta$ ever vanish, and the eigenvalues of Hecke eigenforms are equidistributed with respect to the Sato--Tate distribution). It is hence worth noting that although the cuspidal part $f_Q$ appearing in the decomposition \eqref{eqn:thetaQdecompose} for $\theta_Q$ is in some sense ``mysterious'', the Siegel--Weil formula \eqref{eqn:SiegelWeil} states that a certain natural linear combination of the cusp forms occurring from multiple theta functions indeed vanishes. Viewing these as ``arbitrary'' cusp forms, the linear dependence of these functions is itself extraordinary, yielding to the conclusion that there is a deep connection between these theta functions that relies on the underlying structure of the quadratic forms.
\end{enumerate}
\end{rem}
Note that since $E_{Q_0}$ is independent of the choice of $Q\in \operatorname{gen}(Q_0)$, the Siegel--Weil average  \eqref{eqn:SiegelWeil} implies that if the $n$th coefficient of $E_{Q_0}$ is positive then $r(Q,n)$ is positive for at least one $Q\in\operatorname{gen}(Q_0)$; Tartakowsky's result may be interpreted as the statement that for $n$ sufficiently large the representations are equidistributed across all of the isometry classes. One often refers to $\sum_{Q\in \mathcal{G}(Q_0)}w_Q^{-1}$ as the \begin{it}mass\end{it} of the genus of $Q$. The coefficients of the weighted average in \eqref{eqn:SiegelWeil} was also shown by Siegel to be equal to a product of local densities at each prime $p$, which were in turn shown to be equal to certain $p$-adic integrals and may also written as limits of $r(Q,np^r)$ divided by a power of $p^r$. These local densities were computed by a number of authors in many special cases, culminating in explicit computations of Yang \cite{Yang}. 

For $\ell$ even, Walling \cite{Walling} investigated explicit formulas for the Siegel--Eisenstein series of the degree $m$ theta series coming from representing quadratic forms of rank $m$. Walling's formulas are given in terms of an explicit linear combination of Siegel--Eisenstein series. For the $m=1$ case considered in this paper, Walling's formulas give $E_Q$ as an explicit linear combination of the Eisenstein series defined at each cusp and the constants in her linear combinations are products over the primes dividing $N$ of certain Gauss sums. 

Due to divergence of the naive definitions of the Eisenstein series defined at each cusp, Walling's formulas (for $m=1$) only hold for $\ell\geq 6$; in general, one requires the co-dimension $\ell-m$ to be sufficiently large for absolute convergence of the Eisenstein series, and Walling requires $m<\frac{\ell}{2}-1$ in particular. 

In rank $\ell=3$, Jones \cite[Theorem 86]{Jones} found explicit formulas for the Fourier coefficients of $E_Q$ as sums of class numbers of imaginary quadratic orders; these sums are often called \begin{it}Hurwitz class numbers\end{it}. 

For the rest of this paper, we restrict ourselves to the very special case $\ell=2$ and also restrict ourselves to the case when the discriminant $\Delta:=b^2-4ac<0$ of $Q(x,y):=ax^2+bxy+cy^2$ is fundamental (here $a,b,c\in\Z$ and we denote the corresponding quadratic form by $[a,b,c]$). 

We now fix a fundamental discriminant $\Delta<0$ once and for all and our goal is to use the connection between quadratic forms and the arithmetic of the ideal class group to apply algebraic techniques that yield explicit formulas for $E_Q$ in this case. Recall that Gauss \cite{Gauss} constructed a composition law for the set $H=H(\Delta)$ of $\SL_2(\Z)$-equivalence classes of integral binary quadratic forms of discriminant $\Delta$ which makes $H$ into a group. This group law can be transparently described via a correspondence with the class group of ideals in the ring of integers of the corresponding imaginary quadratic field. The connection between quadratic forms and ideals moreover implies that the automorphs of binary quadratic forms are in one-to-one correspondence with units in the ring of integers $\cO_K$ of the imaginary quadratic field $K=\Q(\sqrt{\Delta})$. Hence $w_Q$ is entirely determined by $\Delta$ and given by 
\begin{equation}\label{eqn:wdef}
w=w_{\Delta}:= \begin{cases} 4 & \text{if $\Delta=-4$} \\
6 &  \text{if $\Delta = -3$} \\
2 & \text{if $\Delta<-4$.}
\end{cases}
\end{equation}
Since $w$ is independent of $Q$, \eqref{eqn:SiegelWeil} simplifies in the case of binary quadratic forms to 
\begin{equation}\label{eqn:Ebinary}
E_{Q_0} = \frac{1}{|\mathcal{G}(Q_0)|}\sum_{Q\in \mathcal{G}(Q_0)}\theta_Q.
\end{equation}
Instead of summing over a given genus as later investigated by Siegel, Gauss considered the sum over the entire class group (see \cite[p.42]{Zagier}). Letting $Q_h$ denote a representative of the class $h \in H$, Gauss showed that for $n>0$
\begin{equation} \label{gauss}
\sum_{h \in H} r(Q_h,n) = w \sum_{t|n} \left(\frac{\Delta}{t}\right).
\end{equation}
Taking the generating function of both sides of \eqref{gauss}, for $n>0$ one recognizes the right-hand side as the $n$th Fourier coefficient of an Eisenstein series. Since the left-hand side is a sum of $\vert H \vert$ modular forms in $M_1\left(\Gamma_0(-\Delta), \left(\frac{\Delta}{\cdot}\right) \right)$, it must be a modular form in the same space; each $Q_h$ represents $0$ precisely once and there are  $\vert H \vert$ such $h$, yielding the constant term of the generating function
\begin{equation} \label{dirichlet}
\frac{1}{w} \sum_{h \in H} \Theta_h = \frac{ \vert H \vert }{w} + \sum_{n \geq 1} \left( \sum_{t|n} \left( \frac{\Delta}{t} \right) \right) q^n,
\end{equation}
where $\Theta_h:=\theta_{Q_h}$. From \eqref{dirichlet} (and the determination of the constant term of a modular form from its other Fourier coefficients), we recover Dirichlet's class number formula by equating the constant terms 
\[
 \vert H \vert =\frac{w \sqrt{-\Delta}}{2\pi} L\left(1, \left( \frac{\Delta}{\cdot } \right) \right)
\]
from the Fourier expansion of the weight one Eisenstein series (cf. e.g. \cite[Theorem 15.1.1]{Lang}). Here and throughout $L(s,\Delta)$ denotes the special value at $s$ of the analytic continuation of the Dirichlet $L$-function associated with the imaginary quadratic field $\Q(\sqrt{\Delta})$ for a fundamental discriminant $\Delta<0$.

Siegel's formula \eqref{eqn:Ebinary}, however, claims that not only the sum of theta functions taken over the whole class group is an Eisenstein series, but also a subsum taken over classes in a single genus is an Eisenstein series as well. Specifically, let
\[
G:=H/H^2
\]
be the group of genera of binary quadratic forms. Since the genera may be written as cosets of a finite group, each $g\in G$ has the same size $|\mathcal{G}(Q_h)|=|g|=|H^2|$, and we may write \eqref{eqn:Ebinary} in the form
\rm
\[
E_g = \frac{1}{|H^2|}\sum_{h \in g} \Theta_h.
\]
Since $|H^2|$ is an invariant for the discriminant, we relegate ourselves to computing $|H^2|E_g$ in this paper.

We first instead consider twisted versions of \eqref{gauss}. Specifically, emulating Gauss's \eqref{gauss} and the associated formula \eqref{dirichlet}, for a character $\chi$ of $H$ one may consider the associated sums 
\begin{align*}
r_{\chi}(n)&:=\sum_{h \in H}\chi(h) r(Q_h,n),\\
E_{\chi}&:=\frac{1}{w}\sum_{h\in H} \chi(h)\Theta_{h}.
\end{align*}
We restrict ourselves to genus characters $\chi$ on $G$ and note that in this case
\begin{align}
\nonumber r_{\chi}(n)&=\sum_{h \in H}\chi(h) r(Q_h,n)=\sum_{g\in G} \chi(g)\sum_{h\in g} r(Q_h,n),\\
\label{eqn:Echi} E_{\chi}&=\frac{|H^2|}{w}\sum_{g\in G} \chi(g)E_{g}.
\end{align}
The genus characters are in one-to-one correspondence with pairs of discriminants $d>0$ and $D<0$ for which $dD=\Delta$ (see, for example, \cite[Section I.2]{GKZ}), and we denote the corresponding genus character by $\chi_{d,D}$. On the other hand, every such pair of discriminants defines a weight one Eisenstein series  $E_{d,D}$ on
$\Gamma_0(|\Delta|)$ with character $\left(\frac{\Delta}{n}\right)$ 
defined by (cf. \cite[Sections 4.5--4.8]{DiamondShurman})  
\begin{equation}\label{eqn:Edef}
E_{d,D}(\tau):=\frac{1}{2}\delta_{d=1}L(0,D)+ \sum_{n\geq 1} \sum_{t\mid n} \left(\frac{d}{n/t}\right)\left(\frac{D}{t}\right) q^n
\end{equation}
whose Fourier coefficients (except the constant term) are manifestly simple divisor sums. We will show in Section \ref{blackboard} that the two Eisenstein series constructed from a pair of discriminants $(d,D)$ coincide.
\rm
\begin{thm}\label{thm:genustwist}
For discriminants $d>0$ and $D<0$ such that $dD=\Delta$ is fundamental, we have 
\[
E_{\chi_{d,D}} = E_{d,D}.
\]
\end{thm}
Using orthogonality of characters, one obtains formulas for the Eisenstein series associated to the individual genera, giving an explicit formula for Siegel's Theorem. 
\begin{thm}\label{thm:Siegel}
For $g\in G$ we have that 
\[
E_{g} = \frac{w}{|H|} \sum_{\chi_{d,D}\in G^*} \chi_{d,D}(g) E_{d,D},
\]
and hence $E_g$ is in the space of Eisenstein series in particular.
\end{thm}

\begin{rem}
\noindent

\noindent
\begin{enumerate}[leftmargin=*]
\item
Since the $n$th Fourier coefficients of $E_{d,D}$ are explicitly given as sums over the divisors of $n$, Theorem \ref{thm:Siegel} yields an explicit formula for the coefficients of $E_g$ after computing $\chi_{d,D}(g)$. However, $\chi_{d,D}(g)$ is explicitly given for example in Gross--Kohnen--Zagier \cite[Section I.2]{GKZ} in terms of integers represented by $g$. 
\item
The assumption that $\Delta$ is fundamental does not appear to be easily overcome with our techniques due to a missing connection between the quadratic forms and ideals, although Siegel's theorem holds in this case and hence we know that sums like \eqref{gauss} and more generally $r_{\chi}(n)$ for a genus character $\chi$ would indeed yield coefficients of Eisenstein series. One would need a generalized version of \eqref{gauss} in order to generalize in this direction.

\item
Fei Xu has pointed out that the result in this paper may be a special case of formulas from \cite{Fei} because the orthogonal group is a torus in this special case, which may make the formulas simple to compute and write in this form. However, some work seems to be involved in translating between the language of orthogonal groups presented there and the language of class groups presented here. It may be interesting to compute the representation masses using the results from \cite{Fei} in this special case and compare the answer with Theorem \ref{thm:Siegel} in order to see if one obtains non-trivial formulas by computing the coefficients in these two different ways.
\end{enumerate}
\end{rem}
For completeness, we present a "modular forms" proof of \eqref{gauss} in Section \ref{hecke}. In fact, it is not difficult to prove \eqref{gauss} in a straightforward way (see e.g. \cite[Chapter 3, Section 8, Exercise 18]{BS}, or Theorem 5.2 at \url{http://math.ou.edu/~jcook/LaTeX/massformula.pdf}, \cite[Proposition 2.1]{Williams}, \cite{KW}). However, we believe that our approach is interesting enough to justify presenting it here; the proof is also given here so that the arguments are self-contained.

The paper is organized as follows. In Section \ref{hecke}, we present a proof of  \eqref{gauss} by applying the Hecke operators to the corresponding generating function and investigating split, ramified, and inert primes in the ring of integers for $K=\Q(\sqrt{\Delta})$. In Section \ref{blackboard}, we study twisted sums of representations, proving Theorems \ref{thm:genustwist} and \ref{thm:Siegel}.

\section*{Acknowledgements}
The authors thank Tomoyoshi Ibukiyama, Lynn Walling, and Fei Xu for helpful comments following a presentation of the ideas from this paper at the Conference on the Arithmetic Theory of Quadratic Forms held in Seoul National University, Seoul, Korea in January, 2019. 


\section{Hecke operators and formula \eqref{gauss} } \label{hecke}
In this section, we present a well-known proof of \eqref{gauss}, using the fact that $\sum_{h\in H}r(Q_h,n)$ may be interpreted as the number of ideals of norm $n$. To some extent, we follow the lines from \cite[p. 41]{Zagier} 
or \cite[Korollar to Satz 1, p. 100]{Zagier_z} (from which \eqref{gauss} is concluded in \cite[(11)]{Zagier_z}).
 We let
\[
\mathcal{L}(s):=\sum_{n\geq 1} \left(\sum_{h \in H} r(Q_h,n) \right) n^{-s},
\]
and we want to show that $\mathcal{L}(s)$ has a certain type of Euler product. 

Equivalently (see e.g. \cite[Section 4.2]{Zagier}), 
the modular form
\[
\sum_{h \in H} \Theta_h = \sum_{n\geq 1} \left(\sum_{h \in H} r(Q_h,n) \right) q^n \ \in M_1\left(\Gamma_0(-\Delta), \left(\frac{\Delta}{\cdot}\right) \right)
\]
is a Hecke eigenform whose eigenvalue for a prime $p$ is $\left(1+\left(\frac{\Delta}{p}\right)\right)$.
\rm
The following Proposition implies \eqref{gauss} immediately.
\begin{prop} \label{hecke_action}
Let 
\[
a(n):=\sum_{h \in H} r(Q_h,n)
\]
and, for a prime $p$,
\[
b(n):=a(pn) + \left(\frac{\Delta}{p}\right) a(n/p)
\]
with the usual convention $a(n/p)=0$ if $p \nmid n$.
We have that
\[
b(n)=\left(1+\left(\frac{\Delta}{p}\right)\right)a(n).
\]
\end{prop}

\begin{rem}
\noindent

\noindent
\begin{enumerate}[leftmargin=*]
\item
We do not actually need to know that $\sum_{h \in H} \Theta_h$ is a modular form in order to derive \eqref{gauss} from Proposition \ref{hecke_action}.
\item
By induction, we also conclude that 
\[
a(n)=a(1)\sum_{t\mid n} \left(\frac{\Delta}{t}\right). 
\]
Hence $a(n)$ is multiplicative up to normalization by the factor $a(1)$. 
One easily checks that $a(1)=w$, as the number of representations of $1$ is the same as the number of elements of norm $1$ (i.e., the number of units). 

Another interpretation of $w$ is as the number of automorphs of the binary quadratic forms of discriminant $\Delta$. 
Indeed, $w=w_Q$ coincides with the factor divided to weight each term in the average \eqref{eqn:SiegelWeil}, but $w_Q$ is not an invariant of the discriminant in dimension higher than two. 
\end{enumerate}
\end{rem}

It is convenient to use certain standard operators $U=U_p$ and $V=V_p$ and Hecke operators (we, of course, have in mind $M_1\left(\Gamma_0(-\Delta), \left(\frac{\Delta}{\cdot}\right) \right)$). For a formal power series 
\[
f=\sum_{n\geq 1} c(n) q^n
\]
 and a prime $p\in \Z$ we set
\begin{equation}\label{eqn:Tpdef}
f \vert U:= \sum_{n \geq 1} c(pn) q^n, \hspace{3mm} f \vert V:= \sum_{n \geq 1} c(n) q^{pn}, \hspace{3mm} \text{and} \hspace{3mm}
f \vert T_p := f \vert U +  \left(\frac{\Delta}{p}\right) f \vert V.
\end{equation}

Let $K=\Q(\sqrt{\Delta})$ be the imaginary quadratic extension. 

Proposition \ref{hecke_action} in this 
notation 
 states that, for every prime $p$,
\begin{equation}\label{eqn:reform2.1}
\left. \left( \sum_{h \in H} \Theta_h \right) \right\vert T_p = \left(1+\left(\frac{\Delta}{p}\right)\right) \sum_{h \in H} \Theta_h = 
\begin{cases}
2  \sum_{h \in H} \Theta_h & \text{if $p$ splits in $K$} \\
0  & \text{if $p$ is inert in $K$} \\
\sum_{h \in H} \Theta_h & \text{if $p$ is ramified in $K$}, \\
\end{cases}
\end{equation}
and this is the statement that what we are going to prove.

For the proof of Proposition \ref{hecke_action} we need some standard facts about the correspondence between ideals in $\cO_K$ and quadratic forms which we recall now. For every non-zero ideal $I \subset \cO_K$ there exists a basis $<\alpha,\beta>$ which generates $I$ as a rank two free $\Z$-module (a lattice on the complex plain). The norm of an ideal, $N(I)$ is, by definition, its index 
\[
N(I)=[\cO_K:I] = \vert \cO/I \vert.
\] 
If $I=(u)$ is principal, then $N(I)=N(u)=u\bar{u}$ where the bar denotes complex conjugation.

We now associate a quadratic form to an ideal $I$. We choose a basis  $<\alpha,\beta>$ of $I$ as a $\Z$ module, and
set
\[
Q_I=\frac{N(\alpha x + \beta y)}{N(I)} = ax^2+bxy+cy^2.
\]
By \cite[Theorem 2.7.4]{BS}, we have that $b^2-4ac=\Delta$, and $a,b,$ and $c$ are integers with gcd$(a,b,c)=1$. Moreover,
the map $I \mapsto Q_I$ establishes an isomorphism between the group of classes of ideals in $\cO_K$ (modulo principal ideals) and elements of $H=H(\Delta)$.

The above discussion allows us to write 
\[
\Theta_h = \sum_{m \in I_h} q^{N(m)/N(I_h)}
\]
where $I_h$ is any ideal in the class $h \in H$.
\begin{proof}[Proof of Proposition \ref{hecke_action}]

Note first that, by the uniqueness of decomposition of ideals into products of prime ideals in $\cO_k$,  for $m\in I_h$ we have 
\begin{equation} \label{eqn:common}
p\mid \frac{N(m)}{N(I_h)} \hspace{3mm} \text{if and only if} \hspace{3mm} m\in \mathfrak{q} I_h \hspace{3mm} \text{for some prime ideal $\mathfrak{q}\mid (p)$}. 
\end{equation}
Here  
\[
\mathfrak{q} = \begin{cases}
\text{$\mathfrak{p}$ or $\mathfrak{p}'$} & \text{if $\left(\frac{\Delta}{p}\right)=1$ and $(p)=\mathfrak{p}\mathfrak{p}'$}, \\
(p) & \text{if }\left(\frac{\Delta}{p}\right)=-1, \\
\mathfrak{p} & \text{if $\left(\frac{\Delta}{p}\right)=0$ and $(p)=\mathfrak{p}^2$}.
\end{cases}
\]
We see from \eqref{eqn:common} that the image of $\Theta_h$ under the operator $U_p$ may be naturally written as a sum over prime ideals dividing $(p)$. Although one may unify the argument in this way, we split it into the three cases of $p$ being split, ramified, or inert for clarity of presentation.

We begin with the case that $p$ is split. In this case, in order to verify \eqref{eqn:reform2.1} it suffices to check that 
\begin{equation} \label{hecke_split}
\Theta_h \vert T_p = \Theta_{h\p} + \Theta_{h\p'}.
\end{equation}
Indeed, \eqref{hecke_split} will imply that 
\[
\sum_{h \in H} \Theta_h \vert T_p = \sum_{h \in H\p} \Theta_h + \sum_{h \in H\p'} \Theta_h = 2 \sum_{h \in H}\Theta_h
\]
because $H=H\p=H\p'$ as a set.

\begin{rem} \label{rem_spl}
Note that in fact $h\p$ and $h\p'$ belong to the same genus because in the genera group $g=g^{-1}$ for every $g \in G$,
and $h\p  h\p' = h^2 (p) = id_G$. We therefore actually may derive from \eqref{hecke_split} a finer identity for splitting prime $p$:
\[
E_g \vert T_p = 2E_{g\p}=2E_{g\p'}.
\]
\end{rem}
In order to prove \eqref{hecke_split}, we note that by \eqref{eqn:common}, $N(m)/N(I_h)$ is divisible by $p$ if and only if $\mathfrak{p}\mid m$ or $\mathfrak{p}'\mid m$. However, if the intersection $(p)=\mathfrak{p}\mathfrak{p}'$ divides $m$, then we have double-counted the contribution from this ideal, and inclusion-exclusion yields
\rm
\[
\Theta_h | U_p = \left( \sum_{m \in I_h \p} + \sum_{m \in I_h \p'} \right) q^{N(m)/pN(I_h)} - \sum_{m \in I_h \p \cap  I_h \p'} q ^ {N(m)/pN(I_h)}.
\]
Recalling the definition of $T_p$ from \eqref{eqn:Tpdef}, in order to finish our proof of \eqref{hecke_split} it now suffices to note that the first sum of sums in the right is $ \Theta_{h\p} + \Theta_{h\p'}$ while the subtracted sum equals $\Theta_h \vert V_p$ because $I_h \p \cap  I_h \p' = (p)I_h$ (by the unique decomposition again) while $N(pm)=p^2N(m)$.

We next consider the case that $p=\mathfrak{p}^2$ is ramified. In this case $T_p=U_p$ and \eqref{eqn:reform2.1} follows from 
\begin{equation} \label{hecke_ramified}
\Theta_h \vert U_p = \Theta_{h\p}.
\end{equation}
Indeed, exactly as in the splitting case, \eqref{hecke_ramified} will imply that 
\[
\sum_{h \in H} \Theta_h \vert T_p = \sum_{h \in H\p} \Theta_h = \sum_{h \in H}\Theta_h
\]
because $H=H\p$ as a set. 

\begin{rem} \label{rem_ram}
Note that, similarly to Remark \ref{rem_spl} above, for ramified primes, $T_p$ permutes the series $E_g$:
\[
E_g \vert T_p = E_{g\p}
\]
for every genus $g \in G$.
\end{rem}

In order to verify \eqref{hecke_ramified}, observe that by \eqref{eqn:common} and the fact that $p=N(\mathfrak{p})$, we have 
\[
\Theta_h | U_p =  \sum_{m \in I_h \p} q^{N(m)/pN(I_h)}  =  \sum_{m \in I_h \p} q^{N(m)/N(I_h \p)} = \Theta_{h\p}.
\]

We finally consider the case that $p$ is inert. In this case \eqref{eqn:reform2.1} is equivalent to 
\begin{equation} \label{hecke_inert}
\Theta_h \vert T_p =0
\end{equation}
for every $h \in H$. By \eqref{eqn:common} and the fact that $I_h$ and $I_h(p)$ belong to the same class in $H$ and $N((p))=p^2$, we obtain
\rm
\[
\Theta_h | U_p =  \sum_{m \in I_h (p)} q^{N(m)/pN(I_h)} =  \sum_{m \in I_h (p)} q^{pN(m)/N(I_h(p))} = \Theta_{I_h(p)} \vert V = \Theta_{I_h} \vert V_p.
\]
Since $ \left(\frac{\Delta}{p}\right)=-1$, the claim now follows from the definition of $T_p$ \eqref{eqn:Tpdef}.

\end{proof}

\section{Twists with genus character} \label{blackboard}

In this section, we generalize \eqref{dirichlet} to prove Theorem \ref{thm:genustwist}  and obtain Theorem \ref{thm:Siegel} as a corollary. Recall that (see \cite[Section I.2]{GKZ})) the value of the genus character $\chi_{d,D}(g)$ on $g \in G$ is defined by 
\[
\chi_{d,D}(g):=\left(\frac{d}{r}\right),
\]
where $r$ is any integer represented by at least one $Q\in g$ and $(d,r)=1$ unless $Q=[a,b,c]$ satisfies $(a,b,c,d)>1$ (in which case no such $r$ exists and we choose $r=0$). We also recall the definition \eqref{eqn:wdef} of $w=w_{\Delta}$ and  the definition \eqref{eqn:Edef} of the Eisenstein series $E_{d,D}$. 
\begin{proof}[Proof of Theorem \ref{thm:genustwist}]
The $d=1$ case of Theorem \ref{thm:genustwist} is precisely \eqref{dirichlet} because $\chi_{1,\Delta}(g)=1$ and 
\[ 
E_{\chi_{1,\Delta}}= \frac{\vert H^2 \vert}{w} \sum_{g\in G} E_g = \frac{1}{w} \sum_{g\in H/H^2}\sum_{h\in g} \Theta_{h} =  \frac{1}{w} \sum_{h\in H} \Theta_h,
\]
while 
\[
E_{1,\Delta}(\tau)=\frac{1}{2}L(0,\Delta)+\sum_{n\geq 1} \sum_{t\mid n} \left(\frac{\Delta}{t}\right) q^n. 
\]
Writing $r_g(n):=\sum_{h\in g} r(Q_h,n)$, we have 
\[
E_g(\tau)=\frac{1}{|H^2|}\sum_{n\geq 0} r_g(n) q^n. 
\]
Since $\chi_{d,D}(g)= \left(\frac{d}{n}\right)$ for any $n$ represented by $Q$ for $Q\in g$, we have $\chi_{d,D}(g)r_g(n) = \left(\frac{d}{n}\right) r_g(n)$ for every $n$ and hence 
\[
|H^2|\sum_{g\in G} \chi_{d,D}(g) E_g(\tau) = \sum_{n\geq 0}\sum_{g\in G}\chi_{d,D}(g) r_g(n)q^n= \sum_{n\geq 0}\left(\frac{d}{n}\right)\sum_{g\in G} r_g(n)q^n.
\]
For $n=0$ we have $\left(\frac{d}{n}\right)=0$ unless $d=1$. Moreover, from the $d=1$ case we know that for $n\geq 1$
\[
\sum_{g\in G} r_g(n)=w\sum_{t\mid n}\left(\frac{\Delta}{t}\right).
\]
Therefore, for $d>1$ we have 
\begin{align*}
E_{\chi_{d,D}}(\tau)&=\frac{|H^2|}{w}\sum_{g\in G} \chi_{d,D}(g) E_g(\tau) =\sum_{n\geq 1}\left(\frac{d}{n}\right)\sum_{t\mid n}\left(\frac{\Delta}{t}\right) q^n\\
&= \sum_{n\geq 1}\left(\frac{d}{n}\right)\sum_{t\mid n}\left(\frac{d}{t}\right)\left(\frac{D}{t}\right) q^n= \sum_{n\geq 1}\sum_{t\mid n}\left(\frac{d}{n/t}\right)\left(\frac{D}{t}\right) q^n=E_{d,D}(\tau),
\end{align*}
yielding the claim.
\end{proof}

We finally conclude the special case of Siegel's mass formula. Let $G^*$ denote the dual group for $G$. Since $G=H/H^2$ is a $2$-group, every genus character $\chi \in G^*$ is real and, as noted before Theorem \ref{thm:Siegel} (see \cite[Section I.2]{GKZ} for further details), $G^*$ precisely coincides with the set of characters $\chi_{d,D}$. 
While this is well-known, we supply a short proof for the convenience of the reader.
\begin{lemma}\label{lem:G*}
For a fundamental discriminant $\Delta<0$ and $G=H(\Delta)/H(\Delta)^2$, we have 
\[
G^*=\left\{\chi_{d,D}: dD=\Delta,\text{ and }d>0\right\}.
\]
\end{lemma}
\begin{proof}
One begins with a famous theorem of Gauss which states that $|G|=2^{t-1}$ (see \cite[Theorem 3.8.5]{BS}), where $t$ is the number of pairwise distinct prime divisors of $\Delta$. Since $|G^*|=|G|$, it hence suffices to construct $2^{t-1}$ pairwise distinct genus characters $\chi_{d,D}$, which is equivalent to making $2^{t-1}$ distinct choices of positive discriminants $d\mid \Delta$. As usual, we decompose 
\[
\Delta=(-1)^{a_3+1}2^{\alpha}\prod_{\substack{p\mid \Delta\\ p\equiv 1\pmod{4}}} p  \prod_{\substack{p\mid \Delta\\ p\equiv 3\pmod{4}}} (-p),
\]
where $(-1)^{a_3+1}2^{\alpha}\in \{1,-4,8,-8\}$ and $a_j$ is the number of $p\equiv j\pmod{4}$ such that $p\mid \Delta$. Then for a vector $\bm{\delta}$ indexed by the primes $p$ dividing $2\Delta$ with $\delta_p\in \{0,1\}$, we set 
\[
d_{\bm{\delta}}:=\left((-1)^{a_3+1}2^{\alpha}\right)^{\delta_{2}}\prod_{\substack{p\mid \Delta\\ p\equiv 1\pmod{4}}} p^{\delta_p}  \prod_{\substack{p\mid \Delta\\ p\equiv 3\pmod{4}}} (-p)^{\delta_p}. 
\]
If $2\nmid \Delta$, then we set $\delta_2=1$. There are hence $2^t$ such choices of $d_{\bm{\delta}}$, but we have the additional restriction that $d_{\delta}>0$. Note that if every odd $p\mid \Delta$ satisfies $p\equiv 1\pmod{4}$, then since $\Delta<0$ we must have $(-1)^{a_3+1}2^{\alpha}\in \{-4,-8\}$ and the condition $d>0$ is equivalent to $\delta_2=0$ in this case. On the other hand, if there is an odd prime $p_0\equiv 3\pmod{4}$ dividing $\Delta$, then there is a bijection between those $d_{\bm{\delta}}>0$ and $d_{\bm{\delta}'}<0$ formed by changing $\delta_{p_0}\mapsto 1-\delta_{p_0}$. In both cases, one of the $\delta_p$ (either $\delta_2$ or $\delta_{p_0}$ is determined by the other components of $\bm{\delta}$ and the condition $d_{\bm{\delta}}>0$. We see that there are hence precisely $2^{t-1}$ such choices of $d_{\bm{\delta}}$ . 
\end{proof}

We are now ready to conclude the main theorem.

\begin{proof}[Proof of Theorem \ref{thm:Siegel}]

Since $G=H/H^2$, thinking of $g\in G$ as a representative in $H$ we have $g=g_0$ if and only if $g\in g_0 H^2$, and hence
\[
E_{g_0} =  \sum_{g \in G} E_{g}\delta_{g \in g_0 H^2}.
\]
Note next that since $g_0^{-1}=g_0$ in $G$, we have that $g\in g_0 H^2$ if and only if 
\[
g_0 g =g_0^{-1}g \in H^2.
\]
Thus $\delta_{g\in g_0 H^2}=\delta_{g_0g\in H^2}$ and we may use the orthogonality relation 
\[
\frac{1}{|G|}\sum_{\chi\in G^*}\chi(g)=\delta_{g\in H^2}
\]
and the total multiplicativity of the characters to rewrite 
\[
E_{g_0} =   \frac{1}{|G|} \sum_{g \in G} E_{g} \sum_{\chi\in G^*} \chi(g_0)\chi(g).
\]
By Lemma \ref{lem:G*}, we have $G^*=\{\chi_{d,D}: d>0,D<0, dD=\Delta\}$, so the above may be rewritten as 
\[
E_{g_0}= \frac{1}{|G|} \sum_{g \in G} E_{g} \sum_{\chi_{d,D}\in G^*} \chi_{d,D}(g_0) \chi_{d,D}(g).
\]
We then interchange the sums and plug in \eqref{eqn:Echi} to obtain
\[
E_{g_0}= \frac{1}{|G|} \sum_{\chi_{d,D}\in G^*} \chi_{d,D}(g_0) \sum_{g \in G} \chi_{d,D}(g)  E_{g}= \frac{1}{|G|} \sum_{\chi_{d,D}\in G^*} \chi_{d,D}(g_0) \frac{w}{|H^2|}E_{\chi_{d,D}}.
\]
We may then plug in Theorem \ref{thm:genustwist} to obtain 
\[
E_{g_0}=\frac{w}{|G|\cdot |H^2|}  \sum_{\chi_{d,D}\in G^*}\chi_{d,D}(g_0) E_{d,D}.
\]
Finally, we use 
\[
|H|=|H/H^2|\cdot |H^2|=|G|\cdot |H^2|
\]
to conclude that 
\[
E_{g_0}=\frac{w}{|H|}  \sum_{\chi_{d,D}\in G^*}\chi_{d,D}(g_0) E_{d,D}.
\]

\end{proof}

\end{document}